\documentclass[article,10pt, reqno]{amsart}
\usepackage[letterpaper,body={16.0cm,22.0cm}, mag=1000]{geometry}

\usepackage{amsmath,amsfonts,amssymb,upgreek}

\usepackage[breaklinks]{hyperref}
\usepackage{amscd}
\usepackage{mathtools}

\theoremstyle{thmrm}
\usepackage{tikz}
\usetikzlibrary{calc,decorations.markings}
\usepackage{tikz-cd}
\usetikzlibrary{cd}
\usepackage{graphicx}

\theoremstyle{plain}

\newtheorem{thm}{Theorem}[section]
\newtheorem{lemma}[thm]{Lemma}
\newtheorem{prop}[thm]{Proposition}
\newtheorem{cor}[thm]{Corollary}

\newtheorem{defn}[thm]{Definition}

\theoremstyle{definition}

\newcommand{\dlabel}[1]{\ifmmode \text{\ttfamily \upshape [#1] } \else
{\ttfamily \upshape [#1] }\fi \label{#1}}

\newcommand{\B}{\operatorname{B} }
\newcommand{\C}{\operatorname{C} }

\newcommand{\Ho}{\operatorname{H} }

\newcommand{\Z}{\operatorname{Z} }

\newcommand{\Id}{\operatorname{Id}}

\newcommand{\Aut}{\operatorname{Aut} }
\newcommand{\Autb}{\operatorname{Autb} }

\newcommand{\Ext}{\operatorname{Ext} }

\newcommand{\Inn}{\operatorname{Inn} }

\newcommand{\Ker}{\operatorname{Ker} }
\newcommand{\Out}{\operatorname{Out} }
\newcommand{\Ann}{\operatorname{Ann} }
\newcommand{\Soc}{\operatorname{Soc} }
\newcommand{\IM}{\operatorname{Im} }

\title{Extensions and Well's type exact sequence of skew braces}

\author{Nishant}

\address{School of Mathematics, Harish-Chandra Research Institute, HBNI,
Chhatnag Road, Jhunsi, Allahabad - 211 019, INDIA}

\email{nishant@hri.res.in}


\begin{document}

\maketitle

\begin{abstract}
In this article,  we give a description of the split exact sequences of left skew braces.  We define a free action of the second cohomology group of a left skew brace $H$ by $Ann(I)$ on $Ext_{\alpha}(H, I)$ and show that this action becomes transitive if $I$ is a trivial skew brace.  We also generalize the Well's type exact sequence for extensions by the trivial skew brace.
\end{abstract}
\section{Introduction}
A  triple $(E,+,\circ)$, where $(E,+)$ and $(E, \circ)$ are groups is said to be a left skew brace if 
$$a \circ (b+c)=a\circ b-a+a \circ c$$
holds for all $a,b,c \in E$, where $-a$ denotes the inverse of $a$ in $(E, +)$.  In 2007,  Rump \cite{WR07} introduced classical braces to study involutive and non-degenerate solutions of the Yang-Baxter equation.  Later,  Guarnieri and Vendramin \cite{GV17}  generalized this concept to skew brace  to study the non-degenerate solution of the Yang-Baxter equation, which is further generalized to semi-braces by Catino, Colazzo, and Stefanelli in \cite{FMP21} to study non-bijective solutions of the Yang-Baxter equation.  In  \cite{DG16},  Ben David and Ginosar investigated extensions of bijective $1$-cocycles.  Carter, Elhambadi and Satio in \cite{CES} developed homology and cohomology theories for solution sets of the Yang-Baxter equations. Different homology theories for various structures related to solutions of the Yang-Baxter equations were investigated extensively by Lebed and Vendramin \cite{LV17}. Cohomology and extensions of  linear cycle sets with trivial actions is studied by Lebed and Vendramin \cite{LV16}.   Recently generalized  by Jorge A. Guccione,  Juan J. Guccione and Christian Valqui  \cite{GG21} to non trivial actions. Various type of products like matched product,  semi-direct product, asymmetric product has been defined for the solutions of Yang-Baxter equation [see \cite{DB18}, \cite{BCJO19},  \cite{CCS},  \cite{CCS1},  ,\cite{CCS2},   \cite{WR08}].  In \cite{NMY},  M. K. Yadav and author developed the theory of skew brace extensions for skew brace extensions by an abelian group  and developed the Well's type exact sequence for skew braces.  This work can be thought as a generalization of  \cite{DB18},  \cite{LV17} at the level of extensions.  The fundamental exact sequence of Wells for groups was introduced by C Wells in \cite{W71}.  The fundamental exact sequence of Wells with various applications is carried out in all fine details in \cite[Chapter 2]{PSY18}. A similar exact sequence for cohomology, extensions and automorphisms of quandles was constructed in \cite{BS20}.
 In this paper, we define a new product for the skew braces and construct few examples.  We give constructions for skew braces  similar to that of group theory and generalize the Well's type exact sequence for the trivial skew brace.

\section{Preliminaries}

An algebraic structure $(E, + , \circ)$ is said to be a \emph{left skew brace} if $(E,  +)$ and $(E, \circ)$ are a group and the following compatibility condition holds:
\begin{equation}\label{bcomp}
a \circ (b  + c ) = a \circ b -a + a\circ c
\end{equation}
for all $ a, b , c \in E$,  where $-a$ denotes the inverse of $a$ with respect to `$+ $'. 
Notice that the identity element $0$ of $(E,  +)$ coincides with the identity element  of $(E, \circ)$.

For a left skew brace $E$ and $a \in E$, define a map $\lambda_a : E \to E$ by
$$\lambda_a(b) = -a + (a \circ b)$$
for all $b \in E$. The automorphism group  of a group $G$ is denoted by $\Aut(G)$. 
We have the following  result for skew braces .
\begin{lemma}
Let $(E,+, \circ)$ be a left skew brace, then for each $a \in E$, the map $\lambda_a$ is an automorphism of $(E, +)$ and the map $\lambda : (E, \circ) \to \Aut(E, +)$ given by $\lambda(a) = \lambda_a$ is a group homomorphism.
\end{lemma}

A  sub skew brace $I$ of a left  skew brace $E$ is said to be a \emph{left ideal} of $E$ if $\lambda_a(y) \in I$ for all $a \in E$ and $y \in I$.  A left ideal of $E$ is said to be an \emph{ideal} if $(I, \circ)$ is  a normal subgroup of $(E, \circ)$. The Socle of a skew brace $E$ is defined as $\Soc(E) = \Ker \lambda$  $\cap $ $\Z (E, +)$, where $\Z(E, +)$ represents the centre of the group $(E, +)$ and the annihilator of $E$ is defined as $\Ann(E)=\Soc(E) \cap \Z(E, \circ)$.

The following is an easy but important observation, which will be used several times in what follows.
\begin{lemma}
Let $E$ be a left skew brace. Then for all $a, b \in E$, the following hold:

(i) $a + b = a \circ \lambda^{-1}_a(b)$.

(ii) $a \circ b = a + \lambda_a(b)$.
\end{lemma}

Let $E_1$ and $E_2$ be two  left skew braces. A map $f : E_1 \to E_2$ is said to be a \emph{skew brace homomorphism} if $f(a + b) = f(a) + f(b)$ and $f(a\circ b) = f(a) \circ f(b)$ for all $a, b \in E_1$.  A one-to-one and onto skew brace homomorphism from $E_1$ to itself is called an \emph{automorphism} of $E_1$. The \emph{kernel} of a homomorphism $f : E_1 \to E_2$ is defined to be the subset $\{a \in E_1 \mid f(a) = 0\}$ of $E_1$. It turns out that $\Ker(f)$, the kernel of $f$, is an ideal of $E_1$. The set of all skew brace automorphisms of a left skew brace $E$, denoted by $\Autb(E)$, is a group.

Let $H$ and $I$ be two left skew  braces. By an \emph{extension} of $H$ by $I$, we mean  a left skew brace $E$ with an exact sequence 
$$\mathcal{E} := 0 \to I \stackrel{i}{\to}  E \stackrel{\pi}{\to} H \to 0,$$ 
where $i$ and $\pi$ are injective and surjective brace homomorphisms, respectively.   Thereafter,  we denote the image of $y$ under $i$ by $y$ itself for all $y \in I$.  A set map $s : H \to E$ is called a \emph{set-theoretic section} of $\mathcal{E}$ if $\pi(s(h)) = h$ for all $h \in H$ and $s(0) = 0$. The abbreviation `st-section' will be used for `set-theoretic section' throughout. We call $\mathcal{E}$ to be split exact sequence of skew braces if there exist a st-section of $\mathcal{E}$ which is a skew brace homomorphism. 

\section{split extensions of skew brace}

Let $H$ and $I$ be two left skew braces. Let $ \mu: (H, +) \rightarrow Aut(I, +)$,  $\sigma : (H, \circ) \rightarrow Aut(I, \circ)$ be anti-homomorphisms, and $\nu: (H, \circ) \rightarrow Aut (I, +)$ be a homomorphism.  Let $\mu, \sigma $ and $\nu$ satisfy the following compatibility  condition
\noindent \begin{align}\label{SE}
\nu_{h_1\circ(h_2 + h_3)}(\sigma_{h_2 + h_3}(\nu^{-1}_{h_1}(y_1)) \circ \nu^{-1}_{h_2 + h_3}(\mu_{h_3}(y_2)+ y_3))& = \mu_{-h_1+(h_2 \circ h_3)}(\nu_{h_1 \circ h_2}(\sigma_{h_2}(\nu^{-1}_{h_1}(y_1)) \circ \nu^{-1}_{h_2}(y_2))-y_1) \notag \\
& +  \nu_{h_1 \circ h_3}(\sigma_{h_3}(\nu^{-1}_{h_1}(y_1)) \circ \nu^{-1}_{h_3}(y_3)) 
 \end{align}
 
for all $y_1, y_2 , y_3 \in I $ and $ h_1, h_2, h_3 \in H$.

\begin{thm}
Let  $H$ and $I$ be two skew braces with $(\nu ,\mu,\sigma),$ as defined above and satisfying \eqref{SE}, then the operations 
\begin{align}
 (h_1, y_1) + (h_2,  y_2)&=(h_1 + h_2, \mu_{h_1}(y_1) + y_2), \label{sb+} \\
 (h_1, y_1) \circ (h_2,  y_2)&=(h_1 \circ h_2, \nu_{h_1 \circ h_2}(\sigma_{h_2}(\nu^{-1}_{h_1}(y_1)) \circ \nu^{-1}_{h_2}(y_2)) \label{sbcirc}
\end{align}
define a left skew brace structure on $ H \times I $.
\end{thm}
\begin{proof}
It is easy to check that the given operations define group structure on $H \times I$ and the condition \eqref{bcomp}  follows from the compatibility condition of $(\nu, \mu, \sigma)$.   
\hfill $\Box$

\end{proof}

We call this structure a  \emph{split semi-direct product} of $H$ by $I$ with respect to the triplet $( \nu,\mu, \sigma)$ and denote it by $(H, I, \nu,\mu, \sigma)$.

\begin{lemma}
Let $(H, I, \nu,\mu, \sigma)$ be  split semi-direct product of $H$ by $I$ with respect to some triplet $(\nu,\mu,\sigma)$.  Then the following short exact sequence of skew braces
$$ \mathcal{E} := 0 \to I \stackrel{i}{\to}  (H, I, \nu, \mu,\sigma) \stackrel{\pi}{\to} H \to 0 $$
 splits, where $i$ and $\pi$ are natural injection and projection respectively.
\end{lemma}

\begin{proof}
It is easy to check that the map $ s: H \rightarrow (H, I, \mu, \sigma,\nu)$ given by $s(h)=(h,0)$  is both  a homomorphism of skew braces and  a st-section of $\mathcal{E}$ simultaneously. 
\hfill $\Box$

\end{proof}

\begin{thm}
Let $\mathcal{E} := 0 \to I \stackrel{i}{\to}  E \stackrel{\pi}{\to} H \to 0$ be a split short exact sequence of skew braces.  Then $E$ is a split semi-direct product  of $H$ by $I$.
\end{thm}

\begin{proof}
Let the short exact sequence $\mathcal{E} := 0 \to I \stackrel{i}{\to}  E \stackrel{\pi}{\to} H \to 0$ split.  Then there exists a st-section $s :H \rightarrow E$,  which is also a skew brace homomorphism.  Define $\mu: H \longrightarrow Aut(I, +)$,  $\sigma : H \rightarrow Aut(I, \circ)$, and $\nu : H \rightarrow Aut(I, +) $ by
\begin{align}\label{actions}
\nu_h(y) & =-s(h) + (s(h) \circ y),\notag\\
\mu_h(y) & =-s(h)+ y+ s(h),  \\
\sigma_h(y) & =s(h)^{-1} \circ y \circ s(h).\notag
\end{align}
Since $E$ is a skew brace,  we have 
\begin{equation}\label{sbc1}
(s(h_1) + y_1) \circ \big{(}s(h_2) + y_2 + s(h_3) +y_3 \big{)} = (s(h_1)+y_1) \circ (s(h_2) + y_2) - (s(h_1) + y_1) + (s(h_1) + y_1) \circ (s(h_3) + y_3). 
\end{equation}
Using \eqref{sbc1} and linearity of $s$ in `$+$' and `$\circ$',  we can easily establish that $(\nu, \mu, \sigma)$ satisfies \eqref{SE}.  Hence we have semi-direct product  $(H,I,\mu,\sigma,\nu)$.  We know that every element $ x \in E$ can be uniquely written as $x=s(h)+ y$.  Define $\phi : E \rightarrow (H,I,\mu,\sigma,\nu)$ by $\phi(s(h)+ y)=(h, y)$.  Then $\phi$ is an isomorphism of skew braces and the diagram
$$\begin{CD}
 0 @>i>> I @>>> E @>{{\pi} }>> H  @>>> 0\\
 &&  @V{\text{Id}}VV @V{\phi}VV @ VV{ \text{Id}}V \\
 0 @>i'>> I @>>> (H,I,\nu,\mu,\sigma) @>{{\pi^\prime} }>> H @>>> 0
\end{CD}$$
commutes,  where $i^\prime$ and $\pi^\prime$ are natural injection and projection,  respectively.  This completes the proof.

\hfill $\Box$

\end{proof}

\section{Examples}

In this section we provide some examples of split semi direct product of skew braces.   We have used GAP to compute $\nu, \mu$ and $\sigma$.

\noindent \textbf{Example 1}  Let $\mathbb{Z}$ and $\mathbb{C}$ be trivial skew braces,  Define $\nu, \mu, \sigma : \mathbb{Z} \rightarrow Aut(\mathbb{C})$ by $\nu_1(x)=\mu_1(x)=\sigma_1(x)=-x$. Using  \eqref{sb+} and \eqref{sbcirc},  we can define a skew brace structure on $\mathbb{Z} \times \mathbb{C}$ by  
\begin{align*}
(l, \ y_1)+(m, \ y_2) & = \big(l+m,   \ (-1)^{m+n}y_{1}+y_2\big),\\
(l, \ y_1)\circ(m, \ y_2) & = \big(l+m,  \ y_1+(-1)^{m+n}y_2\big).
\end{align*}

\noindent \textbf{Example 2}  Let $H=D_{2n}= \langle a,b\hspace{.1cm}|\hspace{.1cm} a^{2n}=b^2=e,  bab=a^{-1}\rangle$ and $I=\mathbb{Z}_p$ be trivial skew braces, where $D_{2n}$  and $\mathbb{Z}_p$ denotes dihederal group of order $4n$ and cyclic group of order $p$ respectively.  Define $\nu, \mu, \sigma : D_{2n} \rightarrow Aut(\mathbb{Z}_p)$ by $\nu_a(x)=\mu_a(x)=\sigma_a(x)=-x$ and $\nu_b(x)=\mu_b(x)=\sigma_b(x)=-x$. Using  \eqref{sb+} and \eqref{sbcirc},  we can define a skew brace structure on $D_{2n} \times \mathbb{Z}_p$ by  
\begin{align*}
(a^{i}b^{j}, y_1)+(a^{m} b^{n},  y_2) & = \big(a^{i}b^{j}a^{m} b^{n},   \ y_2+(-1)^{m+n}y_{1}\big),\\
(a^{i}b^{j}, y_1)\circ(a^{m} b^{n},  y_2) & = \big(a^{i}b^{j}a^{m} b^{n},  \ y_1+(-1)^{i+j}y_2\big).
\end{align*}
If we have trivial skew brace $H=D_n$, where $n$ is odd and $I$ be the same as above,  then we can define $\nu, \mu, \sigma : D_{n} \rightarrow Aut(\mathbb{Z}_p)$ by $\nu_a(x)=\mu_a(x)=\sigma_a(x)=x$ and $\nu_b(x)=\mu_b(x)=\sigma_b(x)=-x$.

\

\noindent \textbf{Example 3} Let $H=\mathbb{Z}_8$ be trivial skew brace and $I=\mathbb{Z}_3 \times \mathbb{Z}_2$  be skew brace of order $6$ defined in \cite{EM20} by the following operations
\begin{align*}
 (n,m)+(s,t) & =(n+2^{m}s, m+t),\\
(n,m) \circ (s,t) & =(2^{t}n+2^{m}s,m+t).
\end{align*}
We have $(I,+)=\langle  (1,0), (0,1)  \rangle    \cong S_3$ and $(I, \circ)=\langle (1,1) \rangle \cong \mathbb{Z}_6$. We take $\mu_{a}(n,m) = (n,m)$,  $ \sigma_{a}(n,m)=(n,m)^{-1}=(2n,m), $  $\nu_{a}(n,m)=(2n,m) $,  for all $(n,m) \in I$, where $a$  is a generator of $H$.  Hence the additive group of skew brace structure on $(H, I, \nu, \mu, \sigma)$ is just direct product of their respective additive groups and multiplicative group is given as follows
\begin{eqnarray*}
(a^k, (n,m)) \circ (a^l,(s,t))=(a^{k+l},  ((2n,m) \circ (2s,t)^{k}).
\end{eqnarray*} 
\

\noindent \textbf{Example 4} Let $H$ be brace of order $4$ defined in \cite{DB15} by $(H,+)=\mathbb{Z}_2 \times \mathbb{Z}_2$,  $(H, \circ)=\langle(0,1) \rangle \cong \mathbb{Z}_4$ and $I$ be a brace such that $(I,+)=\mathbb{Z}_4$,  $(I,\circ)=\langle 1,2 \rangle \cong \mathbb{Z}_2 \times \mathbb{Z}_2 $. Consider $\mu_x=\Id$ for all $x \in H$,  where $\Id$ denotes the identity mapping on $I$ and $\nu_{(0,1)}(x)=\sigma_{(0,1)}(x)=-x$. Then split semi- direct product of $I$ by $H$ is given by the skew brace with additive group as direct product of $H$ and $I$ and multiplicative group structure is given as follows
\begin{eqnarray*}
\big((0,1)^k, l) \circ ((0,1)^n,  m\big)=\big((0,1)^{k+n},  l + (-1)^{k}m+(-1)^{n}lm\big).
\end{eqnarray*}
\

\noindent \textbf{Example 5} Let $H$ be the  brace of order $8$ with additive group $\mathbb{Z}_8$ having Socle of order $2$ and $(H,\circ)=\langle 1,2 \rangle \cong \mathbb{Z}_4 \times \mathbb{Z}_2$  defined in \cite{DB15} and $I$ be brace of order $4$ as defined in Example 3. Then we have total $8$ different split semi-direct products of $I$ by $H$,  interestingly with $\mu_x=Id$ for all $x \in H$ in all cases. We list few cases
\begin{enumerate}
\item[(i)]
$\nu_1(x)=x^{-1}, \hspace{.1cm} \nu_2(x)=x \hspace{.1cm} and \hspace{.1cm} \sigma_h(x)=x \hspace{.1cm}for \hspace{.1cm}all \hspace{.1cm} x \in I, \hspace{.1cm} h \in H. $\\

\item[(ii)]
$\nu_1(x)=x^{-1}, \hspace{.1cm} \nu_2(x)=x^{-1} \hspace{.1cm} and \hspace{.1cm} \sigma_2(2)=3 ,\hspace{.1cm} \sigma_1(x)=x \hspace{.1cm}  for \hspace{.1cm}all \hspace{.1cm} x \in I.$ 

\end{enumerate}

\section{general extensions of skew braces}

Let $(H, +,  \circ)$ and $(I, + , \circ) $ be two skew braces,  $\mathcal{E} := 0 \to I \stackrel{i}{\to}  E \stackrel{\pi}{\to} H \to 0$ be an extension of $H$ by $I$. and let $s : H \rightarrow E$ be  an st-section of $\mathcal{E}$.  Corresonding to $s$, consider the pair $(\beta, \tau)$, where $\beta$ and $\tau$ are defined as  
\begin{align}
\beta(h_1, h_2) &:= - s(h_1 + h_2) + s(h_1) + s(h_2),\label{cocycle1 sb}\\
\tau(h_1, h_2) &:= s(h_1 \circ h_2)^{-1} \circ s(h_1) \circ s(h_2).\label{cocycle2 sb}
\end{align} 
It is easy to see that $ \nu, \mu $ and $ \sigma$,  defined in \eqref{actions},  need not be homomorphisms in general, but they satisfy the following identities

\begin{align}
\nu_{h_1 \circ h_2}&=\nu_{h_1}  \nu_{h_2} \lambda^{-1}_{\tau(h_1, h_2)},\label{action1 }\\
 \mu_{h_1 + h_2} &= i^{+}_{-\beta(h_1, h_2)} \mu_{h_2} \mu_{h_1}, \label{action2 }\\
\sigma_{h_1 \circ h_2}&=i^{\circ}_{\tau(h_1, h_2)^{-1}} \sigma_{h_2} \sigma_{h_1}, \label{action3}
 \end{align}
where 
 
\begin{align}
i^{+}_y(z)&:=y+z-y, \\
i^{\circ}_y(z)&:=y \circ z \circ y^{-1},
 \end{align}
are inner automorphisms of $(H,+)$ and $(H, \circ)$, respectively, and $\beta$ and $\tau$ are as defined above in \eqref{cocycle1 sb} and \eqref{cocycle2 sb}.
 
Let $N$ be the smallest normal subgroup of $\Aut(I, +)$ generated by the set $\{ \lambda_y \hspace{.1cm}| \hspace{.1cm} y \in I\}$. Let $ \Inn(I , +)$ and $ \Inn(I, \circ)$ be the inner automorphism subgroups  of $\Aut(I, +)$ and $\Aut(I, \circ)$ respectively. Then we have the maps  $\bar{\nu} : (H, \circ ) \rightarrow \Aut(I, +)/N$, $\bar{\mu} : (H, +) \rightarrow \Aut(I, )/ \Inn(I, +)$ and $\bar{\sigma} : (H, \circ) \rightarrow \Aut(I, \circ)/ \Inn(I, \circ)$ defined by $\nu , \mu$ and $\sigma$ composing with natural projections respectively.  We call the triplet  $\chi:=(\nu, \mu, \sigma)$ satisfying (\ref{action1 }), (\ref{action2 }), and (\ref{action3}) an action of $H$ on $I$ and corresponding triplet $\bar{\chi}:=(\bar{\nu},\bar{\mu}, \bar{\sigma})$ will be called a coupling from $H$ to $I$ corresponding to $(\nu,\mu, \sigma)$.  Let $\chi =(\nu, \mu, \sigma)$ and $\chi^\prime=(\nu^\prime, \mu^\prime, \sigma^\prime)$ be two actions. Then we say that $\bar{\chi}^\prime\approx\bar{\chi}$ if there exists a map $\theta: H \rightarrow I$ such that $\theta(0)=0$ and $\nu^{\prime}_{h}=\nu_{h} \lambda_{\theta(h)}$,  $\mu^\prime_{h}=i^+_{\nu_h(-\theta(h))} \mu_h$ and $\sigma^\prime=i^{\circ}_{\theta(h)^{-1}} \sigma_{h}.$

{\bf{Remark:}} If $\bar{\chi}^\prime\approx\bar{\chi}$ then $\bar{\chi}^\prime =\bar{\chi}$ but converse need not be true.  Note that the map $\theta : H \rightarrow I$ mentioned above need not be unique. 

With this setting, we have
\begin{prop}\label{well-def-act-coc}
Let $0 \to I \stackrel{}{\to} E \stackrel{\pi}{\rightarrow}  H \to 1$ be an extension of a left skew brace $I$ by $H$. Then the following hold:
 
(1) The coupling $\bar{\chi}$ is independent of the choice of an st-section. 
 
(2) Equivalent extensions have the same coupling.

\end{prop}
\begin{proof}
(1) Let $s_1$ and $s_2$ be two st-sections of $\pi$. We know that two sections differ by an element of $I$, hence for an element $h \in H$, there exist $y_h \in I$ such that $s_2(h)=s_1(h)\circ y_h$.  Let $\chi=(\nu, \mu ,  \sigma)$ and $\chi^\prime=(\nu^\prime, \mu^\prime, \sigma^\prime)$ be actions corresponding to $s_1$ and $s_2$ respectively. Define $\theta: H \rightarrow I$ be  $\theta(h)=y_h$.  It can be easily seen that $\bar{\chi}^\prime \approx \bar{\chi}$  using $\theta$ as a required map.
 
(2) Let $E$ and $E^\prime$ be two equivalent extensions. Then there exist a skew brace homomorphism $\phi: E^\prime \rightarrow E$ such that the following diagram commutes
$$\begin{CD}
 0 @>>> I @>>> E^\prime @>{{\pi^\prime} }>> H  @>>> 0\\
 &&  @V{\text{Id}}VV @V{\phi}VV @ VV{ \text{Id}}V \\
 0 @>>> I @>>> E @>{{\pi} }>> H @>>> 0.
\end{CD}$$
Let $s : H \rightarrow E^\prime$ be any st-section of  the extension $E^\prime$.  Then $\phi s : H \rightarrow E$ is a st-section of extension the $E$. Let $\chi=(\nu, \mu, \sigma)$ be actions of $E$ corresponding to $\phi s$ and $\chi^\prime=(\nu^\prime, \mu^\prime, \sigma^\prime)$ be actions of $E^\prime$ corresponding to $ s$, respectively.  Then we have $\nu=\nu^\prime$, $\mu = \mu^\prime$ and $\sigma=\sigma^\prime$.  Hence $\bar{\chi} \approx \bar{\chi^\prime}$ by taking $\theta: H \rightarrow I$ to be $\theta(h)=0$ for all $h \in H$.  As we have already proved that coupling is independent of an st-section,  so this holds for all st-sections of $E$ and $E^\prime$.
\hfill $\Box$

\end{proof}
Let $\Ext(H,I)$ denote the set of  equivalence classes of all  extensions of $H$ by $I$.  Equivalence class of an extension $\mathcal{E} : 0 \to I \to E \to H \to 0$ is denoted by $[\mathcal{E}]$. As a consequence of the preceding proposition, it follows that each equivalence class of extension of $H$ by $I$ admits a unique coupling $\bar{\chi}=(\bar{\nu},\bar{\mu}, \bar{\sigma})$ corresponding to actions $\chi=(\nu, \mu, \sigma)$ of $H$ on $I$.  Let $\Ext_{(\bar{\nu},\bar{\mu}, \bar{\sigma})}(H, I)$ denote the equivalence class of those extensions of $H$ by $I$ whose corresponding coupling is $(\bar{\nu},\bar{\mu}, \bar{\sigma})$.  We can easily establish

\begin{cor}\label{cor 1}
 $\Ext(H, I) = \bigsqcup_{(\bar{\nu},\bar{\mu}, \bar{\sigma})} \Ext_{(\bar{\nu},\bar{\mu}, \bar{\sigma})}(H, I)$.
 \end{cor}

\begin{prop}\label{prop 2}
Let $\mathcal{E}$ be a extension of $H$ by $I$. Then the following hold \\

\begin{itemize}
\item[1)]
Let $s$ be an st-section of $\mathcal{E}$. Then the pair $(\beta, \tau)$ corresponding to $s$ together with action defined in \eqref{actions} satisfies

\begin{equation}\label{cocycle 1}
\beta(h_1, h_2+h_3)+\beta(h_2, h_3)-\beta(h_1+h_2, h_3)-\mu_{h_3}(\beta(h_1, h_2))=0,
\end{equation}
and 
\begin{equation}\label{cocycle 2}
\tau(h_1, h_2 \circ h_3) \circ \tau(h_2, h_3) \circ \tau(h_1 \circ h_2, h_3)^{-1} \circ (\sigma_{h_3}(\tau(h_1, h_2)))^{-1}=0.
\end{equation}

\item[2)]
Let $s_1$ and $s_2$ be two st-sections of $\mathcal{E}$, and let $(\beta_1, \tau_1)$ and $(\beta_2, \tau_2)$ be the pairs corresponding to $s_1$ and $s_2$, respectively. Let $(\prescript{}{1}{\nu}, \prescript{}{1}{\mu}, \prescript{}{1}{\sigma})$  and $(\prescript{}{2}{\nu}, \prescript{}{2}{\mu}, \prescript{}{2}{\sigma})$ be actions corresponding to $s_1$ and $s_2$, respectively. Then there exists a map $\theta : H \rightarrow I$ such that 
$$
s_2(h)=s_1(h) \circ \theta(h)=s_1(h) + \prescript{}{1}{\nu}_{h}(\theta(h)),
$$ 
\begin{equation}\label{equi 1}
\prescript{}{1}{\nu}_{h_1+h_2}(-\theta(h_1+h_2))+\beta_1(h_1, h_2)+ \prescript{}{1}{\mu}_{h_2}(\prescript{}{1}{\nu}_{h_1}(\theta(h_1)))+\prescript{}{1}{\nu}_{h_2}(\theta(h_2))=\beta_2(h_1, h_2),
\end{equation}
and
\begin{equation}\label{equi 2}
\theta(h_1 \circ h_2)^{-1} \circ \tau_1(h_1, h_2) \circ \prescript{}{1}{\sigma}_{h_2}(\theta(h_1)) \circ \theta(h_2)=\tau_2(h_1, h_2),
\end{equation}
for all $h, h_1, h_2 \text{ and } h_3 \in H$.\\

\item[3)]
Let $\mathcal{E}_1$ and $\mathcal{E}_2$ be two equivalent extensions of $H$ by $I$, and let $s_1$ and  $s_2$ be st-sections of $\mathcal{E}_1$ and $\mathcal{E}_2$, respectively. Let $(\beta_1, \tau_1)$ and $(\beta_2, \tau_2)$ be the pairs corresponding to $s_1$ and $s_2$, respectively. Then there exists a map $\theta : H \rightarrow I$ satisfying \eqref{equi 1} and \eqref{equi 2}.

\end{itemize}

\end{prop}
\begin{proof}
It is easy to see that (1) and (2) follows directly from definitions. Now we will prove (3). 
Since $\mathcal{E}_1 := 0 \to I \stackrel{}{\to}  E_1 \stackrel{\pi_1}{\to} H \to 0$ and $\mathcal{E}_2:= 0 \to I \stackrel{i}{\to}  E_2 \stackrel{\pi_2}{\to} H \to 0$ are two equivalent extensions, there exists an isomorphism $\phi : E_1 \rightarrow E_2$ such that the following diagram commutes
$$
\begin{CD}
 0 @>>> I @>>> E_1 @>{{\pi_1} }>> H  @>>> 0\\
 &&  @V{\text{Id}}VV @V{\phi}VV @ VV{ \text{Id}}V \\
 0 @>>> I @>>> E_2 @>{{\pi_2} }>> H @>>> 0.
\end{CD}
$$

Let $s_1$ be an st-section of $\mathcal{E}_1$. Then $\phi s_1$ is an st-section of $\mathcal{E}_2$.  Let $(\beta_1, \tau_1)$ and $(\beta^\prime, \tau^\prime)$ be the pairs corresponding to $s_1$ and $\phi s_1$, respectively. By the commutativity of the above diagram, we have $\phi(y)=y$, for all $y \in I$, hence $\beta_1=\beta^\prime$ and $\tau_1=\tau^\prime$. Let $s_2$  be an st-section of $\mathcal{E}_2$ and $(\beta_2, \tau_2)$ be the pair corresponding to $s_2$. Using \eqref{equi 1}, \eqref{equi 2} for $s_2$ and $\phi s_1$, we get the desired result. 
\hfill $\Box$

\end{proof}
\begin{defn}
Let $\chi=(\nu, \mu, \sigma)$ be an action of $H$ on $I$, and $\beta, \tau : H \times H \rightarrow I$ such that $\beta$ and $\tau$ together with $\chi$ satisfies \eqref{cocycle 1} and \eqref{cocycle 2}, respectively. Then the ordered pair $(\beta, \tau)$ is a $2$-cocycle with action $\chi$. 

{\bf{Remark:}} Note that the pair $(\beta, \tau)$ defined by the \eqref{cocycle1 sb} and \eqref{cocycle2 sb} is a $2$-cocycle corresponding to st-section $s$ with action defined by \eqref{actions}. 
\end{defn}

Let  $\mathcal{E}:= 0 \to I \stackrel{}{\to}  E\stackrel{\pi}{\to} H \to 0$ be an skew brace extension of $H$ by $I$. Let $s : H \rightarrow E$ be an st-section of  $\mathcal{E}$. Due to the compatiblity condition of a left skew brace $E$, we have
$$
(s(h_1) \circ y_1) \circ (s(h_2) \circ y_2+s(h_3) \circ y_3)=s(h_1) \circ y_1) \circ (s(h_2) \circ y_2)-(s(h_1) \circ y_1)+(s(h_1) \circ y_1) \circ (s(h_3) \circ y_3).
$$ 
From the above equality we see that the triple $(\nu, \mu, \sigma)$ (defined  in \eqref{actions}) together with $(\beta, \tau)$ satisfy
\begin{eqnarray}\label{parent relation}
 \nu_{h_1 \circ (h_2+h_3)}(\tau(h_1, h_2+h_3) \circ \sigma_{h_2+h_3}(y_1) \circ \nu^{-1}_{h_2+h_3}(\beta(h_2, h_3) +\mu_{h_3}(\nu_{h_2}(y_2))+\nu_{h_3}(y_3))) = A,
\end{eqnarray}
where
\begin{align*}
A = &\beta (h_1 \circ h_3-h_1, h_1 \circ h_3)+\mu_{h_1 \circ h_3}(\beta(h_1 \circ h_2, -h_1)+ \mu_{-h_1}(\nu_{h_1 \circ h_2}(\tau(h_1, h_2) \circ \sigma_{h_2}(y_1) \circ y_2))\\
& -\nu_{h_1}(y_1)- \beta(h_1, -h_1)) +\nu_{h_1 \circ h_3}(\tau(h_1, h_3) \circ \sigma_{h_3}(y_1) \circ y_3),
\end{align*}
for all $h_1, h_2, h_3 \in H$ and $y_1, y_2, y_3 \in I.$

{\bf{Remark:}} Note that if $I$ is an abelian group equipped with trivial skew brace structure, then the above condition will simplify to the condition defined for good triplet of actions in \cite[Pg.5]{NMY}.

 For $\alpha=(\bar{\nu},\bar{\mu}, \bar{\sigma})$ a coupling from $H$ to $I$.  Define
\begin{align*}\label{stZ^2}
\mathcal{Z}^2_{\alpha}(H, I):=\Bigg\{(\chi,\beta, \tau)\hspace{.1cm} \Big| \hspace{.1cm} \hspace{.1cm}\substack{\chi \mbox{ is an action of } H \mbox{ on } I, \hspace{.1cm}  \bar{\chi} \approx \alpha, \hspace{.1cm}  \mbox{and} \hspace{.1cm}(\beta,\tau)\hspace{.1cm} \mbox{ia a 2-cocycle} \hspace{.1cm}\mbox{with action}  \\ \hspace{.1cm} \chi  \hspace{.1cm} \mbox{and satisfy}  \hspace{.1cm} \eqref{parent relation}
  }\Bigg\}.
\end{align*}
Let $(\chi_1, \beta_1, \tau_1)$ and $(\chi_2 \beta_2, \tau_2)$ be two elements of  $\mathcal{Z}^2_{\alpha}(H, I)$, where $\chi_1=(\prescript{}{1}{\nu}, \prescript{}{1}{\mu}, \prescript{}{1}{\sigma})$ and $\chi_2=(\prescript{}{2}{\nu}, \prescript{}{2}{\mu}, \prescript{}{2}{\sigma})$.  We say that  $(\chi_1, \beta_1, \tau_1)$ $\sim$ $(\chi_2, \beta_2, \tau_2)$ if there exits a map $\theta: H \rightarrow I$ such that $\bar{\chi_2} \approx \bar{\chi_1} $ by $\theta$ and $\beta_1$, $\beta_2$ satisfy (\ref{equi 1}), $\tau_1, \tau_2$ satisfy (\ref{equi 2}) with respect to $\theta$. \\

\begin{prop}
The  relation  `$\sim$' defined in above para is an equivalence relation.
\end{prop}

\begin{proof}
Reflexivity is easy to see by taking $\theta:H\rightarrow I$ given by $\theta(h)=0$ for all $h \in H$. Now we will show that the above relation is symmetric. Let $ (\chi_1, \beta_1, \tau_1)$ $\sim$ $(\chi_2, \beta_2, \tau_2)$, where $\chi_i=(\prescript{}{i}{\nu}, \prescript{}{i}{\mu}, \prescript{}{i}{\sigma})$ for $i=1,2$.  We know that there exist a map  $\theta : H \rightarrow I$ such that $\bar{\chi_1} \approx\bar{\chi_2}$ by $\theta$ and  \eqref{equi 1}, \eqref{equi 2}.  Define $\psi:H \rightarrow I$ by $\psi(h)=\theta(h)^{-1}$.  Then we have $\prescript{}{1}{\nu}_h=\prescript{}{2}{\nu}_h \lambda_{\psi(h)}$, and hence $\prescript{}{2}{\nu}_h(-\psi(h))=\prescript{}{1}{\nu}_h(\theta(h))$.  This proves that $\bar{\chi_1}=\bar{\chi_2} $ by $\psi$. Similarly we can prove that $\beta_2, \beta_1$ satisfy (\ref{equi 1}) and $\tau_1, \tau_2$ satisfy (\ref{equi 2}) with respect to $\psi$. Next we prove  transitivity.  Let  $(\chi_1, \beta_1, \tau_1)$ $\sim$ $(\chi_2, \beta_2, \tau_2)$  (by $\theta_1$) and $ (\chi_2, \beta_2, \tau_2)$ $\sim$ $(\chi_3, \beta_3, \tau_3)$ (by $\theta_2$),  where
$\chi_i=(\prescript{}{i}{\nu}, \prescript{}{i}{\mu}, \prescript{}{i}{\sigma})$ for $i=1,2,3$.  We claim that  $\bar{\chi_3} \approx \bar{\chi_1}$ by $\phi:H \rightarrow I $ defined by $\psi(h)=\theta_1(h) \circ\theta_2(h)$. We have 
\begin{eqnarray}\label{first relation}
\prescript{}{2}{\nu_h} =
\prescript{}{1}{\nu_h} \lambda_{\theta_1(h)},
\end{eqnarray}
and 
\begin{eqnarray*}
\prescript{}{3}{\nu}_h =
\prescript{}{2}{\nu_h} \lambda_{\theta_2(h)}.
\end{eqnarray*}

Combining these two equations we have 
\begin{eqnarray*}
\prescript{}{3}{\nu}_h &=&
\prescript{}{1}{\nu_h} \lambda_{\theta_1(h)}\lambda_{\theta_2(h)}\\
&=& \prescript{}{1}{\nu_h} \lambda_{\phi(h)}.
\end{eqnarray*} 
For additive action we have

\begin{eqnarray*}
\prescript{}{2}{\mu}_{h}=i^+_{\prescript{}{1}{\nu_h}(-\theta_1(h))} \prescript{}{1}{\mu}_h
\end{eqnarray*}

\begin{eqnarray*}
\prescript{}{3}{\mu}_{h}=i^+_{\prescript{}{2}{\nu_h}(-\theta_2(h))} \prescript{}{2}{\mu}_h.
\end{eqnarray*}

Combining the above two equations we have 
\begin{eqnarray*}
\prescript{}{3}{\mu}_{h} &=& i^+_{\prescript{}{2}{\nu_h}(-\theta_2(h))}i^+_{\prescript{}{1}{\nu_h}(-\theta_1(h))} \prescript{}{1}{\mu}_h.\\
\end{eqnarray*}
Using (\ref{first relation}) we have 
\begin{eqnarray*}
\prescript{}{3}{\mu}_{h} &=& i^+_{\prescript{}{1}{\nu_h} \lambda_{\theta_1(h)(-\theta_2(h))}}
i^+_{\prescript{}{1}{\nu_h}(-\theta_2(h))} \prescript{}{1}{\mu}_h.\\
\end{eqnarray*}

Finally, using the relation $a+\lambda_a(b)=a \circ b$, we get 
\begin{eqnarray*}
\prescript{}{3}{\mu}_{h} &=& i^+_{\prescript{}{1}{\nu_h}(-\phi(h))}\prescript{}{1}{\mu}_h.
\end{eqnarray*}

Similarly we can prove that $\prescript{}{3}{\sigma}_h=i^{\circ}_{\phi(h)^{-1}} \prescript{}{1}{\sigma_h}$,  which shows that $\bar{\chi_3} \approx \bar{\chi_1}$ by $\phi$.  Next we prove that $\beta_1$ and $\beta_3$ also satisfy (\ref{cocycle 1}) with respect to $\phi$.  We have

\begin{align*}
\prescript{}{1}{\nu}_{h_1+h_2}(-\theta_1(h_1+h_2))+\beta_1(h_1, h_2)+ \prescript{}{1}{\mu}_{h_2}(\prescript{}{1}{\nu}_{h_1}(\theta_1(h_1)))+\prescript{}{1}{\nu}_{h_2}(\theta_1(h_2))=\beta_2(h_1, h_2),\\
\prescript{}{2}{\nu}_{h_1+h_2}(-\theta_2(h_1+h_2))+\beta_2(h_1, h_2)+ \prescript{}{2}{\mu}_{h_2}(\prescript{}{2}{\nu}_{h_1}(\theta_2(h_1)))+\prescript{}{2}{\nu}_{h_2}(\theta_2(h_2))=\beta_3(h_1, h_2).
\end{align*}

Combining these two equations and using the fact that $\prescript{}{1}{\nu}_{h}(\theta_1(h))+\prescript{}{2}{\nu}_{h}(\theta_2(h))=\prescript{}{1}{\nu}_{h}(\theta_1(h)\circ \theta_2(h)) $,  we have 

\begin{align*}
\beta_3(h_1, h_2)& =  \prescript{}{2}{\nu}_{h_1+h_2}(-\theta_2(h_1+h_2))+\prescript{}{1}{\nu}_{h_1+h_2}(-\theta_1(h_1+h_2))+\beta_1(h_1, h_2)+ \prescript{}{1}{\mu}_{h_2}(\prescript{}{1}{\nu}_{h_1}(\theta_1(h_1)))
 \\ 
 &\hspace*{4mm} +\prescript{}{1}{\nu}_{h_2}(\theta_1(h_2))+ \prescript{}{2}{\mu}_{h_2}(\prescript{}{2}{\nu}_{h_1}(\theta_2(h_1)))+\prescript{}{2}{\nu}_{h_2}(\theta_2(h_2))\\
&= \prescript{}{1}{\nu}_{h_1+h_2}(-(\theta_2(h_1+h_2)\circ\theta_1(h_1+h_2)))+\beta_1(h_1, h_2)+ \prescript{}{1}{\mu}_{h_2}(\prescript{}{1}{\nu}_{h_1}(\theta_1(h_1)))\\
& \hspace*{4mm} + \prescript{}{1}{\mu}_{h_2}(\prescript{}{2}{\nu}_{h_1}(\theta_2(h_1)))+ \prescript{}{1}{\nu}_{h_2}(\theta_1(h_2))+\prescript{}{2}{\nu}_{h_2}(\theta_2(h_2))\\
&= \prescript{}{1}{\nu}_{h_1+h_2}(-(\theta_2(h_1+h_2)\circ\theta_1(h_1+h_2)))+\beta_1(h_1, h_2)+ \prescript{}{1}{\mu}_{h_2}(\prescript{}{1}{\nu}_{h_1}(\theta_1(h_1) \circ \theta_2(h_1) ))\\
& \hspace*{4mm} + \prescript{}{1}{\nu}_{h_2}(\theta_1(h_2)\circ \theta_2(h_2))\\
&= \prescript{}{1}{\nu}_{h_1+h_2}(-\phi(h_1+h_2))+\beta_1(h_1, h_2)+ \prescript{}{1}{\mu}_{h_2}(\prescript{}{1}{\nu}_{h_1}(\phi(h_1)))+\prescript{}{1}{\nu}_{h_2}(\phi(h_2)).
\end{align*}

Similar calculation shows that  $\tau_1, \tau_3$ satisfy (\ref{cocycle 2}) with respect to  $\phi$. Hence the relation `$\sim$' is an equivalence relation.
\hfill $\Box$

\end{proof}

Define  $$\mathcal{H}^2_{\alpha}(H, I):=\mathcal{Z}^2_{\alpha}(H, I)/ \sim$$ and denote $[(\chi,\beta, \tau)] \in \mathcal{H}^2_{\alpha}(H, I)$, the equivalence class of $(\chi,\beta, \tau)$. This concept will be used in next section.

\section{action of cohomology group on extensions}
In this section, we define a faithful group action of $\Ho_{N}^2(H, \Z(I))$  \cite[Pg.6]{NMY} on $\Ext_{\alpha}(H, I)$ and we will show that this action is transitive whenever $I$ is  trivial skew brace.
\begin{thm} \label{main}
Let $\alpha$ be a coupling from $H$ to $I$. Then there exists a bijection between $\Ext_{\alpha}(H, I)$ and  $\mathcal{H}^2_{\alpha}(H, I)$.
\end{thm}

\begin{proof}
Define $\phi : \Ext_{\alpha}(H, I) \rightarrow \mathcal{H}^2_{\alpha}(H, I)$ as follows. Let $\mathcal{E} := 0 \to I \stackrel{i}{\to}  E \stackrel{\pi}{\to} H \to 0$ be an extension with coupling $\alpha$.  Fix an st-section $s$, then there exists an action $\chi=(\nu, \mu, \sigma)$  as we defined in \eqref{actions} such that $\bar{\chi}=\alpha$; also we have $(\beta, \tau)$ as we defined in (\ref{cocycle1 sb}) and (\ref{cocycle2 sb}),  together they satisfy equation (\ref{parent relation}).  Set 
$$ 
\phi([\mathcal{E}])=[(\chi, \beta, \tau)].
$$
Then by Proposition \ref{well-def-act-coc} and Proposition \ref{prop 2} the map $\phi$ is well defined.
Next we define a map $\psi : \mathcal{H}^2_{\alpha}(H, I) \rightarrow Ext_{\alpha}(H, I)$ as follows.  Given an element  $(\chi, \beta, \tau)$ of $\mathcal{Z}^2_{\alpha}(H, I),$ we define  binary operations on the set $H \times I$ by setting

(1) $(h_1, y_1)+(h_2, y_2)=(h_1+h_2, \nu^{-1}_{h_1+h_2}(\beta(h_1, h_2)+\mu_{h_2}(\nu_{h_1}(y_1))+\nu_{h_2}(y_2)))$,

(2) $(h_1, y_1) \circ (h_2, y_2)=(h_1 \circ h_2,\tau(h_1, h_2) \circ \sigma_{h_2}(y_1) \circ y_2).$

It is easy to check that (\ref{cocycle 1}) and (\ref{cocycle 2}) gives  the associativity of `$+$' and `$\circ$', respectively, which is enough to see that $(H \times I, +)$ and $(H  \times I , \circ)$ are groups  and (\ref{parent relation}) proves that `$+$' and `$\circ$' defined here satisfy \eqref{bcomp}.  We denote this left skew brace structure by $(H, I , \chi, \beta, \tau)$.  Now Consider the extension 
$$\mathcal{E}(\chi, \beta, \tau) := 0 \to I \stackrel{i}{\to}  (H, I , \chi, \beta, \tau) \stackrel{\pi}{\to} H \to 0,$$ where $i(y)=(0, y)$ and $\pi(h, y)=h$ for all $h \in H$ and $y \in I$.  Define $\psi$ by setting $$\psi([(\chi, \beta, \tau)])= [\mathcal{E}(\chi, \beta, \tau)].$$ 
We show  that the map $\psi$ is well defined. Let $(\chi_1, \beta_1, \tau_1)$ $\sim$ $(\chi_2, \beta_2, \tau_2)$, then there exist a map $\theta : H \rightarrow I $ such that $\bar{\chi_1} \approx \bar{\chi_2}$ by $\theta$ and $\beta_1, \beta_2$ satisfy \eqref{equi 1} and $\tau_1, \tau_2$ satisfy \eqref{equi 2}, respectively.  Define $\zeta: \mathcal{E}(\chi_2, \beta_2, \tau_2) \rightarrow \mathcal{E}(\chi_1, \beta_1, \tau_1)$ given by 
$$
\zeta(h, y)=(h, \theta(h) \circ y).
$$

We have
\begin{align*}
\zeta((h_1, y_1)+(h_2, y_2))=& (h_1+h_2, \ \theta(h_1+h_2) \circ \prescript{}{2}{\nu}^{-1}_{h_1+h_2}(\beta_2(h_1, h_2)+ \prescript{}{2}{\mu}_{h_2}( \prescript{}{2}{\nu}_{h_1}(y_1))+ \prescript{}{2}{\nu}_{h_2}(y_2)))\\
=&(h_1+h_2, \ \theta(h_1+h_2)+ \prescript{}{1}{\nu}^{-1}_{h_1+h_2}(\beta_2(h_1, h_2)+ \prescript{}{2}{\mu}_{h_2}( \prescript{}{2}{\nu}_{h_1}(y_1))+ \prescript{}{2}{\nu}_{h_2}(y_2)))\\
=& (h_1+h_2, \ \prescript{}{1}{\nu}^{-1}_{h_1+h_2}(\beta_1(h_1, h_2)+ \prescript{}{1}{\mu}_{h_2}(\prescript{}{1}{\nu}_{h_1}(\theta(h_1)))+\prescript{}{1}{\nu}_{h_2}(\theta(h_2))\\
&+ \prescript{}{2}{\mu}_{h_2}( \prescript{}{2}{\nu}_{h_1}(y_1))+ \prescript{}{2}{\nu}_{h_2}(y_2)))\\
=& (h_1+h_2, \ \prescript{}{1}{\nu}^{-1}_{h_1+h_2}(\beta_1(h_1, h_2)+ \prescript{}{1}{\mu}_{h_2}(\prescript{}{1}{\nu}_{h_1}(\theta(h_1)))+ \prescript{}{1}{\mu}_{h_2}( \prescript{}{2}{\nu}_{h_1}(y_1))\\
&+ \prescript{}{1}{\nu}_{h_2}(\theta(h_2))+ \prescript{}{2}{\nu}_{h_2}(y_2)))\\
=& (h_1+h_2, \ \prescript{}{1}{\nu}^{-1}_{h_1+h_2}(\beta_1(h_1, h_2)+ \prescript{}{1}{\mu}_{h_2}(\prescript{}{1}{\nu}_{h_1}(\theta(h_1) \circ y_1))\\
&+ \prescript{}{1}{\nu}_{h_2}(\theta(h_2) \circ y_2)))\\
=&(h_1, \ \theta(h_1) \circ y_1)+(h_2, \ \theta(h_2) \circ y_2).
\end{align*}
Which shows that $\zeta$ is linear in `$+$'. Similarly we can show that $\zeta$ is linear in `$\circ$' as well. It is easy to see that $\zeta$ is an isomorphism and the following diagram commutes
$$
\begin{CD}
 0 @>>> I @>>> \mathcal{E}(\chi_2, \beta_2, \tau_2) @>{{\pi_1} }>> H  @>>> 0\\
 &&  @V{\text{Id}}VV @V{\zeta}VV @ VV{ \text{Id}}V \\
 0 @>>> I @>>> \mathcal{E}(\chi_2, \beta_2, \tau_2) @>{{\pi_2} }>> H @>>> 0,
\end{CD}
$$
where $\pi_1$ and $\pi_2$ are natural projections. Hence, $\mathcal{E}(\chi_1, \beta_1, \tau_1)$ and $\mathcal{E}(\chi_2, \beta_2, \tau_2)$ are equivalent extensions. That shows that the map $\psi$ is well defined. It is easy to check that $\psi$ is well-defined and $\psi$ and $\phi$ are inverses of each other. The proof is now complete.
\hfill $\Box$

\end{proof}

The elements of $\mathcal{Z}^2_{\alpha}(H, I)$ are called associated triplets as every element of $\mathcal{Z}^2_{\alpha}(H, I)$ is associated to some extension in view of Theorem \ref{Main thm}.

\begin{thm}\label{action change}
Let $H$ and $I$ be two skew braces and let $(\chi, \beta, \tau) \in \mathcal{Z}^2_{\alpha}(H, I)$ be an associated triplet.  If $\chi^\prime$ is an action of $H$ on $I$ for which  $\bar{\chi} \approx\bar{\chi}^\prime$, then there exist maps $\beta^\prime, \tau^\prime : H \rightarrow I$ such that $(\chi^\prime, \beta^\prime, \tau^\prime)$ is an associated triplet and $[(\chi, \beta, \tau)]= [ (\chi^\prime, \beta^\prime, \tau^\prime)]$.
\end{thm}
 \begin{proof}
 In the view of Theorem \ref{main} there exists an extension $$\mathcal{E}(\chi, \beta, \tau) := 0 \to I \stackrel{i}{\to}  E \stackrel{\pi}{\to} H \to 0,$$ corresponding to the associated triplet $(\chi, \beta, \tau)$. Let $s : H \rightarrow E$ be a st-section inducing $(\chi, \beta, \tau)$. Since $\bar{\chi} \approx\bar{\chi^\prime}$, there exist a map $\theta : H \rightarrow I$ such that $\theta(0)=0$, $\nu^{\prime}_{h}=\nu_{h} \lambda_{\theta(h)}$,  $\mu^\prime_{h}=i^+_{\nu_h(-\theta(h))} \mu_h$ and $\sigma^\prime=i^{\circ}_{\theta(h)^{-1}} \sigma_{h}$. Define an st-section  $s^\prime(h)=s(h) \circ \theta(h)$ for all $ x \in H$. Consequently we have an associated triplet $(\chi_1, \beta^\prime, \tau^\prime)$ corresponding to the st-section $s^\prime$ and it is easy to see that $\chi^\prime=\chi_1$. Hence $[(\chi, \beta, \tau)]=[(\chi_1, \beta^\prime, \tau^\prime)])=[(\chi^\prime, \beta^\prime, \tau^\prime)]$ as $(\chi, \beta^\prime, \tau^\prime)$ and $(\chi_1,  \beta^\prime, \tau^\prime)$ are associated triplet of the  same extension by different st-sections. This completes the proof.
 \hfill $\Box$

\end{proof}
 We now state a theorem analogous  to  \cite[Theorem 3.6]{NMY}.  Let $H$ be a skew brace and $I$ be an abelian group.  
 
 Define
\begin{align*}
 \Z_N^2(H, I)=\Bigg\{(g ,f) \hspace{.1cm} \Big \vert  \hspace{.1cm}g,f:H \times H \rightarrow I, \hspace{.1cm} \substack{ g,f \hspace{.1cm} \mbox{sastisy}\hspace{.1cm} (\ref{cocycle 1}) \hspace{.1cm}\mbox{and}\hspace{.1cm} (\ref{cocycle 2}),\hspace{.1cm} \mbox{respectively, and} \\ \hspace{.1cm}\mbox{vanish on degenerate tupples}}  \Bigg\},
\end{align*}
and $\B_N^2(H, I)$ is the collection of the pairs $(g, f) \in  \Z_N^2(H, I)$ such that there exists a map $\theta$ from $H$ to $I$ with $g= \nu_{h_1+h_2}(-\theta(h_1+h_2))+\mu_{h_2}((\nu_{h_1}(\theta(h_1)))+\nu_{h_2}(\theta(h_2))$ and $f=-\theta(h_1 \circ h_2) + \sigma_{h_2}\theta(h_1)) + \theta(h_2)$.

Put
\begin{align*}
\Z_N^1(H, I)=\Bigg\{\substack{ \theta \hspace{.1cm}\mbox{is a map from} \hspace{.1cm} H \hspace{.1cm} \mbox{to} \hspace{.1cm} I \hspace{.1cm} \mbox{such that }\hspace{.1cm} \theta(h_1 \circ h_2)= \sigma_{h_2}(\theta(h_1)+\theta(h_2) \hspace{.1cm} \\ \mbox{and} \hspace{.1cm} \nu_{h_1+h_2}(\theta(h_1 + h_2))= \mu_{h_2}(\nu_{h_1}(\theta(h_1))+\nu_{h_2}(\theta(h_2))}   \Bigg\},
\end{align*}
the set $\Z_N^1(H, I)$ is called as the set of derivations, and
$$
\Ho^2_N(H, I):=\Z_N^2(H, I)/\B_N^2(H, I)
$$
 is the second cohomology group of $H$ by $I$.

\begin{thm} 
Let $H$ be a skew brace and let $I$ be an abelian group equiped with trivial brace structure. Let $\mathcal{E} := 0 \to I \stackrel{i}{\to}  E \stackrel{\pi}{\to} H \to 0$ be an extension.  Then the coupling and action are  same, and  there is a bijection between $\Ext_{(\nu, \mu, \sigma)}(H, I)$ and $\Ho^2_N(H, I)$. 
\end{thm}

\begin{thm}\label{Main thm}
Let $[\mathcal{E}] \in  \Ext_{\alpha}(H, I)$ and $(\chi, \beta, \tau)$ be an associated triplet of $\mathcal{E}$. Then  for $[(\beta_1, \tau_1)] \in \Ho^2_{N}(H, \Ann(I))$, the operation $$[(\beta_1, \tau_1)] [\mathcal{E}(\chi, \beta, \tau)]= [\mathcal{E}(\chi, \beta_1+\beta, \tau_1+\tau)] $$ defines a free  action of the group $\Ho^2_{N}(H, \Ann(I))$ on the set $\Ext_{\alpha}(H, I)$. If $I$ is trivial skew brace, then this action becomes transtivite. 
\end{thm}

\begin{proof}
It is easy to check that the  action under consideration is well defined.  Let $[\mathcal{E}(\chi, \beta, \tau)] \in Ext_{\alpha}(H, I)$ and $[(\beta_1, \tau_1)] \in  \Ho_N^2(H, \Ann(I))$ be such that $[(\beta_1, \tau_1)] [\mathcal{E}(\chi, \beta, \tau)]=[\mathcal{E}(\chi, \beta, \tau)]$. Then $$[(\chi, \beta_1+\beta, \tau_1+\tau)]=[(\chi, \beta, \tau)],$$ and therefore there exist a map $\theta: H \rightarrow I$ such that $\theta(0)=0$ and $\bar{\chi} \approx\bar{\chi}$ by $\theta$, which implies that $\theta(h) \in \Ann(I)$ for all $h \in H$, and
\begin{equation*}
-\nu_{h_1+h_2}(\theta(h_1+h_2))+\beta(h_1, h_2)+ \mu_{h_2}(\nu_{h_1}(\theta(h_1)))+\nu_{h_2}(\theta(h_2))=\beta(h_1, h_2)+\beta_1(h_1, h_2)
\end{equation*}
and
\begin{equation*}
\theta(h_1 \circ h_2)^{-1} \circ \tau(h_1, h_2) \circ \sigma_{h_2}(\theta(h_1)) \circ \theta(h_2)=\tau(h_1, h_2)+\tau_1(h_1, h_2)
\end{equation*}
for all $h_1, h_2 , h_3 \in H$(using the fact that $\theta(h) \in \Ann(I) $ for all $h \in H$).

 We have 
 \begin{align*}
 \beta_1(h_1, h_2)&= \nu_{h_1+h_2}(-\theta(h_1+h_2))+\mu_{h_2}((\nu_{h_1}(\theta(h_1)))+\nu_{h_2}(\theta(h_2)),\\ 
\tau_1(h_1, h_2)&=-\theta(h_1 \circ h_2) + \sigma_{h_2}\theta(h_1)) + \theta(h_2).
 \end{align*}
Thus $[(\tau_1, \beta_1)]=1$, and hence the action is free.

Let $\mathcal{E}_1$ and $\mathcal{E}_2$ be two elements in $Ext_{\alpha}(H,I)$. Then for $i=1,2$,  $[\mathcal{E}_i]=[\mathcal{E}_i(\chi_i, \beta_i,\tau_i)]$ for some associated trilpet $(\chi_i, \beta_i,\tau_i)$, where $\chi_i= (\prescript{}{i}{\nu}, \prescript{}{i}{\mu}, \prescript{}{i}{\sigma})$.  By Theorem \ref{action change}, we can construct $(\beta^\prime,\tau^\prime)$ such that $(\chi_2, \beta^\prime, \tau^\prime)$ is an associated triplet with
$$ [\mathcal{E}_1(\chi_1, \beta_1, \tau_1)]=[\mathcal{E}_1(\chi_2, \beta^\prime, \tau^\prime)].$$
We set 
$$\beta_3(h_1, h_2)=\beta^\prime(h_1,h_2)-\beta_2(h_1, h_2)$$ 
and 
$$\tau_3(h_1, h_2)=\tau^\prime(h_1, h_2)\circ \tau_2(h_1, h_2)^{-1}.$$
Now $(\beta^\prime, \tau^\prime)$ and $(\beta_2, \tau_2)$ are $2$-cocycles with the same action $\chi_2$.  Then from (\ref{action1 }), (\ref{action2 }) and (\ref{action3}) it follows that $\beta_3(h_1, h_2) \in Z(I,+)$ and $\tau_3(h_1, h_2) \in Soc(I)$ for all $h_1 , h_2 \in H$. If we take $I$ to be trivial skew brace,  then $\Z(I,+)= \Soc(I)=\Ann(I)$. Finally we get $\beta_3, \tau_3 : H \rightarrow Ann(I)$. It is easy to see that $(\beta_3, \tau_3)$ is $2-cocycle$ with respect to the action $\chi_2$ and $(\beta_3, \tau_3)[\mathcal{E}(\chi_2, \beta_2, \tau_2)]=[\mathcal{E}(\chi_2, \beta^\prime, \tau^\prime)]=[\mathcal{E}(\chi_1, \beta_1, \tau_1)]
$. Hence the action is transtive, which completes the proof.
\hfill $\Box$

\end{proof}

As a consequence, we get
\begin{thm} \label{main 1}
Let $H$ be a skew brace and  let $I$ be a trivial skew brace with a fixed coupling $\alpha$.  Then there exists a bijection between $Ext_{\alpha}(H,I)$ and $Ext_{\alpha}(H, \Z(I))$.
\end{thm}

\section{action of automorphism group on extensions}

Throughout this section we consider  $I$ to be a trivial skew brace and $\mathcal{E} := 0 \to I \stackrel{i}{\to}  E \stackrel{\pi}{\to} H \to 0$ be an extension of skew braces.  Then $ \nu :H \rightarrow Aut(I,+)$ as defined in (\ref{actions}),  is independent of the choice of an st-section.  Also $\bar{\nu}=  \nu$ and  $\bar{\mu}:(H,+) \rightarrow \Out(I)$,  $\bar{\sigma}: (H, \circ) \rightarrow \Out(I)$,  where $Out(I)$ represents the group of outer-automorphisms of $I$. For a pair  $(\phi, \theta) \in \Autb(H) \times \Autb(I)$ of skew brace automorphisms and an  extension  
$$\mathcal{E} :  0 \rightarrow I \stackrel{i}{\rightarrow} E \stackrel{\pi}{\rightarrow} H \rightarrow 0$$
of $H$ by $I$,  we can define a new extension
$$\mathcal{E}^{(\phi, \theta)} : 0 \rightarrow I \stackrel{i\theta}{\longrightarrow} E \stackrel{\phi^{-1} \pi}{\longrightarrow} H \rightarrow 0$$
of $H$ by $I$. Thus, for a given $(\phi, \theta) \in \Autb(H) \times \Autb(I)$, we can define a map from $\Ext(H, I)$ to itself given by 
\begin{equation}\label{act1 sb}
[\mathcal{E}] \mapsto  [ \mathcal{E}^{(\phi, \theta)}].
\end{equation}
 If $\phi$ and $\theta$ are identity automorphisms, than obviously $\mathcal{E}^{(\phi, \theta)} = \mathcal{E}$. It is also easy to see that  
$$[\mathcal{E}]  ^{(\phi_1, \theta_1) (\phi_2, \theta_2)}=  \big([\mathcal{E}]^{(\phi_1, \theta_1)}\big)^{(\phi_2, \theta_2)}.$$
We conclude that the association \eqref{act1 sb} gives an action of the group $\Autb(H) \times \Autb(I)$  on the set $\Ext(H, I)$. From Corollary \ref{cor 1} ,  we know that  $\Ext(H, I) = \bigsqcup_{(\nu, \bar{\mu}, \bar{\sigma})} \Ext_{(\nu, \bar{\mu}, \bar{\sigma})}(H, I)$.  \emph{Let $(\nu, \bar{\mu}, \bar{\sigma})$ be an arbitrary but fixed coupling from $H$ to $I$.}  Let $\C_{(\nu, \bar{\mu}, \bar{\sigma})}$ denote the stabiliser of $\Ext_{(\nu, \bar{\mu}, \bar{\sigma})}(H, I)$ in $\Autb(H) \times \Autb(I)$; more explicitly
$$\C_{(\nu, \bar{\mu}, \bar{\sigma})} = \{ (\phi, \theta) \in \Autb(H) \times \Autb(I) \mid \nu_h=\theta^{-1}\nu_{\phi(h)}\theta,  \bar{\mu}_h = \theta^{-1}\bar{\mu}_{\phi(h)}\theta \mbox{ and } \bar{\sigma}_h = \theta^{-1}\bar{\sigma}_{\phi(h)}\theta \}.$$ 
It is easy to see that $\C_{(\nu, \bar{\mu}, \bar{\sigma})}$  is a subgroup of $\Autb(H) \times \Autb(I)$.  For details see \cite[Pg.15]{NMY}.

Next we consider an action of $ \C_{(\nu, \bar{\mu,} \bar{\sigma})}$ on $\Ho^2_N(H, Z(I))$ (same as in \cite{NMY}) by 
\begin{equation} \label{act3 sb}
[(g, f)] \mapsto [\big(g^{(\phi, \theta)}, f^{(\phi, \theta)}\big)],
\end{equation}
where $g^{(\phi, \theta)}(h_1, h_2)=\theta^{-1}(g(\phi(h_1), \phi(h_2))$.  This action of $  \C_{(\nu, \bar{\mu,} \bar{\sigma})}$ on $\Ho^2_N(H, Z(I))$  is same as the action of $  \C_{(\nu, \bar{\mu,} \bar{\sigma})}$ on $\Ext_{(\nu, \bar{\mu}, \bar{\sigma})}(H, I)$  transferred on $\Ho^2_N(H, Z(I))$ through bijection of Theorem \ref{main 1} . Using this action we can define the semi-direct product $\Gamma =   \C_{(\nu, \bar{\mu,} \bar{\sigma})} \ltimes \Ho^2_N(H,Z(I))$.  We wish to define an action of $\Gamma$ on $\Ext_{(\nu, \bar{\mu}, \bar{\sigma})}(H, I)$. For $(c, h) \in \Gamma$ and $[\mathcal{E}] \in \Ext_{(\nu, \bar{\mu}, \bar{\sigma})}(H, I)$, define 
\begin{equation}\label{act4 sb}
[\mathcal{E}]^{(c, h)} = ([\mathcal{E}]^c)^h.
\end{equation}

\begin{lemma}\label{wells2 sb}
The rule  in \eqref{act4 sb}   gives an action of  $\Gamma$ on $\Ext_{(\nu, \bar{\mu}, \bar{\sigma})}(H, I)$.
\end{lemma}

\begin{proof}
The proof follows on the  lines of   \cite[Lemma 5.2]{NMY}.

\end{proof}

 Let $[\mathcal{E}] \in \Ext_{(\nu, \bar{\mu}, \bar{\sigma})}(H, I)$ be a fixed extension.  Since the action of $\Ho^2_N(H,Z(I))$ on $\Ext_{(\nu, \bar{\mu}, \bar{\sigma})}(H, I)$ is transitive and faithful, for  each $c \in \C_{(\nu, \bar{\mu,} \bar{\sigma})}$, there exists a unique element (say) $h_c$  in  $\Ho^2_N(H,Z(I))$ such that  
 $$[\mathcal{E}]^{c} = [\mathcal{E}]^{h_c}.$$
 We thus have a well defined map $ \omega(\mathcal{E}): \C_{(\nu, \bar{\mu,} \bar{\sigma})} \rightarrow \Ho^2_N(H,Z(I))$ given by
 \begin{equation}\label{wells-map sb}
 \omega(\mathcal{E})(c)=h_c
 \end{equation}
 for $c \in \C_{(\nu, \bar{\mu,} \bar{\sigma})}$. 
 
\begin{lemma}\label{wells3 sb}
The map $ \omega(\mathcal{E}): \C_{(\nu, \bar{\mu}, \bar{\sigma})} \rightarrow \Ho^2_N(H,Z(I))$ defined in \eqref{wells-map sb} is a derivation with respect to the action of $\C_{(\nu, \bar{\mu}, \bar{\sigma})}$ on $\Ho^2_N(H,Z(I))$ given in \eqref{act3 sb}.
\end{lemma}

\begin{proof}
The proof follows on the  lines of  \cite[Lemma  5.3]{NMY}.
\end{proof}

Let 
$$\mathcal{E}: 0 \rightarrow I \rightarrow E \overset{\pi}\rightarrow H $$
be an extension of a left skew brace $H$ by a trivial skew brace $I$ such that $[\mathcal{E}] \in \Ext_{(\nu, \mu, \sigma)}(H,I)$.
Let  $\Autb_I(E)$ denote the subgroup of $\Autb(E)$ consisting of all automorphisms of $E$ which normalize $I$, that is,
$$\Autb_I(E) := \{ \gamma \in \Autb(E) \mid \gamma(y) \in I \mbox{ for all }  y \in I\}.$$ 
For $\gamma \in \Autb_I(E)$, set $\gamma_I := \gamma |_I$, the restriction of $\gamma$ to $I$, and $\gamma_H$ to be the automorphism of $H$ induced by $\gamma$. More precisely, $\gamma_H(h) = \pi(\gamma(s(h)))$ for all $h \in H$, where $s$ is an st-section of $\pi$. Notice that the definition of $\gamma_H$ is independent of  the choice of an st-section. Define a map $\rho(\mathcal{E}) :  \Autb_I(E) \rightarrow  \Autb(H) \times \Autb(I)$ by
$$\rho(\mathcal{E})(\gamma)=(\gamma_H, \gamma_I).$$ 
Although $\omega(\mathcal{E})$ is not a homomorphism, but we can still talk about its set theoretic kernel, that is,
$$\Ker(\omega(\mathcal{E})) = \{c \in C_{(\nu, \bar{\mu},\bar{\sigma})} \mid [\mathcal{E}]^c=[\mathcal{E}]\}.$$

\begin{prop}\label{wells4 sb}
For the extension $\mathcal{E}$,   $\IM(\rho(\mathcal{E})) \subseteq \C_{(\nu, \bar{\mu}, \bar{\sigma})}$ and  
$\IM(\rho(\mathcal{E})) = \Ker(\omega(\mathcal{E}))$.
\end{prop}

\begin{proof}
The proof follows on the  lines of \cite[Proposition 5.4]{NMY}.
\end{proof}

Continuing with the above setting, set $\Autb^{H, I}(E) := \{\gamma \in \Autb_I(E) \mid \gamma_I = \Id,  \gamma_H = \Id\}$. Notice that  $\Autb^{H,I}(E)$ is precisely the kernel of $\rho(\mathcal{E})$. Hence, using Proposition \ref{wells4 sb}, we get

\begin{thm}\label{wells5 sb}
Let $\mathcal{E}: 0 \rightarrow I \rightarrow E \overset{\pi}\rightarrow H$ be a extension of a left skew  brace  $H$ by a trivial skew brace $I$ such that $[\mathcal{E}] \in \Ext_{(\nu, \bar{\mu}, \bar{\sigma})}(H,I)$. Then we have the following exact sequence of groups 
$$0 \rightarrow \Autb^{H,I}(E) \rightarrow \Autb_I(E) \stackrel{\rho(\mathcal{E})}{\longrightarrow} \C_{(\nu,\bar{\mu}, \bar{\sigma})} \stackrel{\omega(\mathcal{E})}{\longrightarrow} \Ho^2_N(H,Z(I)),$$
where $\omega(\mathcal{E})$ is, in general,  only a derivation.
\end{thm}

Further we have
\begin{prop}\label{wells6 sb}
Let  $\mathcal{E} :  0 \rightarrow I \rightarrow E \overset{\pi}\rightarrow H$ be an extension of $H$ by $I$ such that $[\mathcal{E}] \in \Ext_{(\nu, \bar{\mu}, \bar{\sigma})}(H, I)$.  Then  $\Autb^{H,I}(E) \cong \Z^1_N(H,Z(I))$.
\end{prop}

\begin{proof}
The map $\psi : \Z^1_N(H,Z(I)) \rightarrow \Autb^{H,I}(E)$ defined by $\psi(\lambda)(s(h) \circ y)= s(h) \circ \lambda(h) \circ y$ is the  required isomorphism. Rest proof follows on the lines of \cite[Proposition 5.6]{NMY}
\end{proof}

We finally get the following Wells' like exact sequence for skew braces.
\begin{thm}\label{wells7 sb}
Let $\mathcal{E}: 0 \rightarrow I \rightarrow E \overset{\pi}\rightarrow H$ be an extension of a left skew brace  $H$ by a trivial skew  brace $I$ such that $[\mathcal{E}] \in \Ext_{(\nu, \bar{\mu}, \bar{\sigma})}(H,I)$. Then we have the following exact sequence of groups 
$$0 \rightarrow \Z^1_N(H,\Z(I)) \rightarrow \Autb_I(E) \stackrel{\rho(\mathcal{E})}{\longrightarrow} \C_{(\nu,\bar{\mu}, \bar{\sigma})} \stackrel{\omega(\mathcal{E})}{\longrightarrow} \Ho^2_N(H,\Z(I)),$$
where $\omega(\mathcal{E})$ is, in general,  only a derivation.
\end{thm}

\section{Acknowledgements}

I am grateful to my supervisor Prof.  M. K.  Yadav for his constant support, comments and suggestions while doing this project.  I would like to thank Prof.  L.  Vendramin for his kind help in writing GAP code.  The author acknowledge Harish-Chandra Research institute for fantastic facilities and for the serene ambience that it facilitates.


\begin{thebibliography}{99}

\bibitem{EM20}
E.~Acri, M.~Bonatto,  \textit{Skew braces of size pq}.  Comm. Algebra \textbf{48}(5) (2020), 1872-1881.


\bibitem{DB15}
D.~ Bachiller, \textit{Classification of braces of order $p^3$},   J. Algebra \textbf{219}(5) (2015),  3568-3603.

\bibitem{DB18}
D.~ Bachiller, \textit{Extensions, matched products, and simple braces},   J. Pure Appl. Algebra \textbf{222} (2018),  1670-1691.


\bibitem{BCJO18}
D.~Bachiller, F.~Cedo, E.~ Jespers and J.~ Okninski, 
\textit{Iterated matched products of finite braces and simplicity; new solutions of the Yang-Baxter equation}. 
Trans. Amer. Math. Soc. \textbf{370} (2018), 4881-4907.

\bibitem{BCJO19}
D.~Bachiller, F.~Cedo, E.~ Jespers and J.~ Okninski,
 \textit{Asymmetric product of left braces and simplicity; new solutions of the Yang-Baxter equation}. Commun. Contemp. Math. \textbf{21} (2019), 1850042, 30 pp.
 
\bibitem{BS20}
V.~Bardakov and M.~Singh, \textit{Quandle cohomology, extensions and automorphisms},   arXiv:2005.08564v1.
 
  \bibitem{DG16}
 N.~ Ben David and  Y.~ Ginosar,  \textit{On groups of I-type and involutive Yang-Baxter groups}. J. Algebra \textbf{458} (2016), 197-206.
 
 
\bibitem{CES}
J.S.~Carter, M.~ Elhamdadi and M.~ Saito, \textit{Homology theory for the set-theoretic Yang-Baxter equation and knot invariants from generalizations of quandles}. Fund.  Math.  \textbf{184} (2004),  31-54.




\bibitem{CCS}
F.~Catino, I.~ Colazzo and P.~ Stefanelli,  \textit{Regular subgroups of the affine group and asymmetric product of braces}, J. Algebra \textbf{455} (2016), 164-182.

\bibitem{CCS1}
F.~Catino, I.~ Colazzo and P.~ Stefanelli,  \textit{The matched product of set-theoretical solutions of the Yang-Baxter equation}, J. Pure Appl. Algebra \textbf{224} (2020), 1173-1194.

\bibitem{FMP21}
F.~Catino, M.~Mazzotta and P.~Stefanelli  \textit{Inverse semi-braces and the Yang-Baxter equation}.  J. Algebra \textbf{573} (2021) 576-619.

\bibitem{CCS2}
F.~Catino, I.~ Colazzo and P.~ Stefanelli, \textit{The matched product of the solutions to the Yang-Baxter equation of finite order},  Mediterr. J. Math. \textbf{17} (2020), Paper No. 58, 22 pp.


\bibitem{GG21}
J.A.~ Guccione,  J.J.~Guccione, and Christian Valqui, \textit{Extensions of Linear Cycle Sets},  	arXiv:2111.07953
\bibitem{GV17}
L.~Guarnieri and L.~ Vendramin, \textit{Skew braces and the Yang-Baxter equation},  Math. Comp. \textbf{86} (2017),  2519-2534.

\bibitem{GAP}
The GAP Group,  $ Groups \hspace{.2cm}  Algorithms \hspace{.2cm} and \hspace{.2cm} Programming,  \hspace{.3cm}version \hspace{.2cm} 4.11.0 \hspace{.2cm} (2020). $ (http://www.gap-system.org)





\bibitem{V16} L.~Vendramin,  \textit{Extensions of set-theoretic solutions of the Yang-Baxter equation and a conjecture of Gateva-Ivanova}. J. Pure Appl. Algebra \textbf{220} (2016), 2064-2076.

\bibitem{LV16}
V.~Lebed and L.~Vendramin, \textit{Cohomology and Extensions Of Braces}. Pacific J. Math.  \textbf{284} (2016),  191-212

\bibitem{LV17}
V.~Lebed and  L.~Vendramin,  \textit{Homology of left non-degenerate set-theoretic solutions to the Yang-Baxter equation}. Adv. Math. \textbf{304} (2017), 1219-1261.


\bibitem{PSY18}
I.B.S.~Passi, M.~ Singh, and M.K.~ Yadav, \textit{Automorphisms of finite groups}. Springer Monographs in Mathematics. Springer, Singapore, 2018. xix+217 pp.



\bibitem{WR07}
W.  ~Rump,  \textit{Braces, radical rings,  and the quantum Yang Baxter equation}. J. Algebra \textbf{307} (2007),  153-170.

\bibitem{WR08}
W.~Rump, \textit{Semidirect products in algebraic logic and solutions of the quantum Yang-Baxter equation}, J. Algebra Appl. \textbf{7} (2008),
471-490.

\bibitem{W71}
C.~Wells,  \textit{Automorphisms of group extensions}. Trans. Amer. Math. Soc. \textbf{155} (1971), 189-194.

\bibitem{NMY}
M.K.~Yadav, Nishant, \textit{Cohomology, Extensions and Automorphisms of Skew Braces}. 	arXiv:2102.12235.







\end{thebibliography}
\end{document}